\documentclass[leqno]{amsart}

\usepackage[english]{babel} 
\usepackage{hyperref} 
\usepackage{srcltx}
\usepackage{color}
\usepackage{amsmath}
\usepackage{graphicx}
\newtheorem{defin}{Definition}[section]
\newtheorem{theorem}{Theorem}[section]
\newtheorem{proposition}{Proposition}[section]
\newtheorem{lemma}{Lemma}[section]

\newcommand{\R}{\mathbb{R}} 
\newcommand{\T}{\mathbb{T}}
\newcommand{\N}{\mathbb{N}} 
\newcommand{\Z}{\mathbb{Z}}

\newcommand{\X}{\times}
\newcommand{\ADN}{A^{\frac{1}{2}}D^{\frac{1}{2}}_N}
\newcommand\D{\partial}

\newcommand{\wen}{w_N}
\newcommand{\ren}{\rho_N}

\numberwithin{equation}{section}

\begin{document}

%\date{\today}

\title[Approximate deconvolution model for the mean Boussinesq
equation]{On the convergence of an approximate deconvolution
  model to the 3D mean Boussinesq equations}

\author[L. Bisconti]{Luca Bisconti}

\address{(Luca Bisconti) Dipartimento di Matematica e Informatica
  ``U.~Dini,'' Universit\`a degli Studi di Firenze, Via S.~Marta~3,
  I-50139, Firenze, Italia}

\email{\href{mailto:luca.bisconti@unifi.it}{luca.bisconti@unifi.it}}

\begin{abstract} In this paper we study a Large Eddy Simulation (LES)
  model for the approximation of large scales of the 3D Boussinesq
  equations. This model is obtained using the approach first described
  by Stolz and Adams, based on the Van~Cittern approximate
  deconvolution operators, and applied to the filtered Boussinesq
  equations.  Existence and uniqueness of a regular weak solution are
  provided. Our main objective is to prove that this solution
  converges towards a solution of the filtered Boussinesq equations,
  as the deconvolution parameter goes to zero.
\end{abstract}

\maketitle


\section{Introduction}

The interactive motion of a passive scalar and a viscous
incompressible 3D fluid is governed by the following Boussinesq
equations:
\begin{equation} \label{eq:Bouss-visc} \left. \begin{array}{l}
      \D_t u + (u \cdot \nabla) u -\nu\Delta u+ \nabla \pi = \theta e_3,\\
      \D_t \theta + u \cdot \nabla \theta = 0,\\
      \nabla \cdot u = 0, \\
      (u, \theta )\vert_{t=0} = (u_0, \theta_0),
    \end{array}\right.
\end{equation}
where $\nu>0$ is the viscosity, $u = (u_1, u_2, u_3)$ is the velocity
field, $\theta$ that may be interpreted physically as a thermal
variable (or a density variable), $\pi$ the cinematic pressure 
and $e_3:= (0, 0, 1)^T$ where
$\{e_1, e_2, e_3\}$ is the canonical basis of $\R^3$. The data $u_0$
and $\theta_0$ are the given initial velocity and density, where $u_0$
is divergence-free.  This system, possibly considered on appropriate
domains, is used as a mathematical model in the description of various
geophysical phenomena and has a relevant role in atmospheric sciences
(see \cite{Majda-2003, McWill-2006, Sal-1998}). Moreover, it has
received considerable attention in mathematical fluid dynamics for
incompressible flows with a number of studies (see, e.g.,
\cite{Be-Spi-2011, Fan-Zhou-2009, Fan-Zhou, Selmi-2012} for 
some recent papers about this subject).

It is well known that neither the current mathematical theory nor the
analytical improvements are sufficient to show the global
well-posedness of 3D-dimensional Navier-Stokes-like equations, namely
the Boussinesq system, which is a coupling between the fluid velocity
$u$ and a density term $\theta$.  In order to overcome the main
difficulties and to perform numerical simulations, many regularization
methods for the 3D-dimensional fluid equations have been proposed.
Let us recall that the main idea behind LES is that of computing
suitable mean values of the considered quantities (see
\cite{Be-Ili-Lay-2006, Cha-Re-2013, Sag-2001}). More precisely, in LES
models for \eqref{eq:Bouss-visc}, approximations $(w, \rho, q)$ of the
means $(\overline{u}, \overline{\theta}, \overline{\pi})$ are
considered, with
\begin{equation*}
  \overline{u}(t, x) = \int G_\alpha (x, y)u(t,y)dy,
  \,\, \overline{\theta}(t, x) = \int G_\alpha(x, y)\theta(y)dy, 
  \,\, \overline{\pi} = \int G_\alpha(x, y)\pi(y)dy,
\end{equation*}
where $\alpha$ is a scale parameter and $G_\alpha$ is a smoothing
kernel such that $G_\alpha \to \delta$ when $\alpha \to 0$, with
$\delta$ the Dirac function. This is a convolution filter and
represents the case that we consider throughout the article.

When we formally filter Equations \eqref{eq:Bouss-visc}, we obtain
what we call the ``mean Boussinesq equations'', i.e.
\begin{equation} \label{eq:Bouss-mean} \left. \begin{array}{l} \D_t
      \overline{u} + \nabla\cdot \overline{(u \otimes u)} -\nu\Delta
      \overline{u}+ \nabla \overline{\pi} =
      \overline{\theta} e_3,\\
      \D_t  \overline{\theta} +  \nabla \cdot \overline{(\theta u)} = 0,\\
      \nabla \cdot \overline{u} = 0, \\
      (\overline{u}, \overline{\theta} )\vert_{t=0} = (\overline{u_0},
      \overline{\theta_0}),
    \end{array}\right.
\end{equation}
where $u\otimes u := (u_1 u , u_2 u , u_3 u )$ and, in the
current case, we supply this problem with periodic boundary conditions.

Here, we consider the Approximate Deconvolution Model (ADM),
introduced by Adams and Stolz \cite{Adam-Stolz, St-Ad-1999,
  St-Ad-2001}, as far as we know. This model uses, roughly speaking,
similarity properties of turbulence and it is defined by approximating
the filtered bi-linear terms as follows:
\begin{equation*}
  \overline{(v \otimes v)} \sim 
  \overline{(D_N (\overline{v}) \otimes D_N (\overline{v}))}
  \textrm{ and } 
  \overline{(\varphi v)} \sim 
  \overline{(D_N (\overline{\varphi}) D_N (\overline{v}))},
\end{equation*}
where $v$ and $\varphi$ play the role of $u$ and $\theta$
respectively, and the filtering operator $G_\alpha$ is defined by the
Helmholtz filter (see, e.g., \cite{Be-Lew-2012, Lew-2009}), with
$\overline{(\,\cdot \,)}= G_\alpha(\, \cdot\,  )$ and
$G_\alpha:=(I-\alpha^2\Delta)^{-1}$. Here, $D_N$ is the deconvolution
operator, which is constructed using the Van~Cittert algorithm (see,
e.g, \cite{Lew-2009}) and is formally defined by
\begin{equation} \label{eq:deconv} D_N := \sum_{n=0}^N (I -
  G_{\alpha})^n \textrm{ with } N\in\N.
\end{equation}

The ADM that we study in this paper is defined by
\begin{equation} \label{eq:bilinear-form} B(w, w) := \overline{D_N(w)
    \otimes D_N(w)}, \,\,\, \mathcal{B}(\rho, w) :=
  \overline{D_N(\rho) D_N(w)},
\end{equation}
and the system that we consider, in the space-periodic setting, is the
following
\begin{equation}
  \label{eq:Bouss-approx}
  \begin{aligned}
    &\D_tw + \nabla \cdot \overline{D_N(w)\otimes D_N(w)} -\nu\Delta w
    + \nabla q = \rho e_3,
    \\
    &\D_t \rho + \nabla \cdot \overline{D_N(\rho) D_N(w)} -\epsilon
    \Delta \rho =0,
    \\
    &\nabla \cdot w  = 0,\\
    &w(0, x) = \overline{u_0}(x),\\
    &\rho(0, t) = \overline{\theta_0}(x),
  \end{aligned}
\end{equation}
where $\epsilon> 0$ is the diffusion coefficient, with $0<\epsilon
<1$.  We are aimed at considering \eqref{eq:Bouss-approx} as an
approximation of \eqref{eq:Bouss-mean} and, somehow, the related
solutions are such that $w \simeq G_\alpha(u)$ and $\rho \simeq
G_{\alpha}(\theta)$. In this model, we take into account the case in
which ``$N$ is large'', and the convolution operator $D_N$ is such
that
\begin{equation*}
  D_N \to G_\alpha^{-1}=A:= I -\alpha^2\Delta, \textrm{ as } 
  N \to +\infty,
\end{equation*}
in a suitable sense (see below for details, see also
\cite{Be-Ca-Lew-2013}).  Then, taking $\epsilon \to 0$ as $N\to
+\infty$, we prove that the system \eqref{eq:Bouss-approx} converges
 to the averaged Equations \eqref{eq:Bouss-mean}, as $N\to+\infty$,
when the scale of filtration $\alpha$ remain fixed.

Given $\theta_0, u_0\in L^2(\T^3)$ with $\nabla\cdot u_0=0$, in the
weak sense, we first show the existence and uniqueness of solutions
$(w_N^\epsilon, \rho_N^\epsilon, q_N^\epsilon)$ to problem
\eqref{eq:Bouss-approx} (cf. Theorem~\ref{preliminary-results}) such
that
\begin{equation*}
  \begin{aligned}
    & w_N^\epsilon, \in L^\infty(0, T; W^{1,2}(\T^3)^3)
    \cap L^2(0, T; W^{2,2}(\T^3)^3),\\
    & \rho_N^\epsilon \in L^\infty(0, T; W^{1,2}(\T^3))
    \cap L^2(0, T; W^{2,2}(\T^3)),\\
    & q_N^\epsilon \in L^2(0, T; W^{1,2}(\T^3)) \cap L^{5/3}(0, T;
    W^{2,5/3}(\T^3)).
  \end{aligned}
\end{equation*}
Later on, we will take $\epsilon$ depending on $N$, and we will write
$(w_N, \rho_N, q_N)$ in place of $(w_N^\epsilon, \rho_N^\epsilon,
q_N^\epsilon)$ as the solution of \eqref{eq:Bouss-approx}. Our main
result reads as follows.
\begin{theorem} \label{main} Let $\theta_0, u_0\in L^2(\T^3)$ with
  $\nabla\cdot u_0=0$ and let $\alpha > 0$.  Assume that $\epsilon\to
  0$ as $N\to +\infty$. Then, from the sequence of solutions $\{(w_N ,
  \rho_N , q_N )\}_{N\in \N}$ to \eqref{eq:Bouss-approx}, one can
  extract a sub-sequence (still labeled $\{(w_N , \rho_N, q_N
  )\}_{N\in \N}$) such that
  \begin{align*}
    & \!\left. \begin{array}{ll} w_N \to w\!\! &\!\!
        \left\{ \begin{array}{l} \!  \textrm{weakly in } L^2(0, T;
            W^{2,2}(\T^3)^3) \textrm{ and } \textrm{weakly}^\star
            \textrm{ in }
            L^{\infty}(0, T; W^{1,2}(\T^3)^3),\\
            \! \textrm{strongly in } L^p(0, T; W^{1,2}(\T^3)^3),\,\,
            \forall 1 \leq p < +\infty,
          \end{array} \right. 
      \end{array} \right. \\
    &\!  \left. \begin{array}{ll}
        \rho_N \to \rho\!\!  & \!\!
        \left\{ \begin{array}{l}
            \! \textrm{weakly in } L^2(0, T; W^{1,2}(\T^3)) \textrm{ and }
            \textrm{weakly}^\star \textrm{ in } 
            L^{\infty}(0, T; W^{1,2}(\T^3)),\\
            \!  \textrm{strongly in } L^2(0, T; L^2(\T^3)),
          \end{array} \right. 
      \end{array} \right. \\
    &\,\,\,\,\,  q_N \to  q \textrm{ weakly in } L^2(0, T; W^{1,2}(\T^3)) \cap 
    L^{5/3}(0, T ; W^{2,5/3}(\T^3)),
  \end{align*}
  with $(w, \rho, q )$ verifying the weak formulation for the
  system \eqref{eq:Bouss-mean}. Further, the following energy
  inequality holds true
  \begin{equation*}
    \frac{1}{2}\frac{d}{dt} \left(\|Aw\|^2 + \|A\rho\|^2\right)
    + \nu\|\nabla
    Aw\|^2 \leq \big(A \rho e_3, Aw \big).
  \end{equation*}
\end{theorem}
\medskip

\noindent$\mathbf{Plan\,\, of\,\, the\,\, paper}$ In
Section~\ref{basics} we recall the notation, we introduce the used
functional spaces and we summarize the main properties of the
deconvolution operator $D_N$.  Next, in
Section~\ref{sec-approximate-model}, we prove the existence and
uniqueness result for the problem \eqref{eq:Bouss-approx} and some
fundamental bounds for the related solutions. Finally, in
Section~\ref{limit}, Theorem~\ref{main} is proved.

\section{Basic facts and notation} \label{basics} In this section we
introduce the functional setting that we will use in the sequel, and
we give the definition and the main properties of the deconvolution
operator.

We denote by $x := (x_1, x_2, x_3) \in \R^3$ a generic point in
$\R^3$. Let be given $L \in \R^\star_+ := \{x \in \R : x > 0\}$, and
define $\Omega := ]0, L[^3 \subset \R^3$.  We put $
{\mathcal{T}}_3 := 2\pi\Z^3/L$ and $\T^3$ is the torus
defined by $\T^3 := \big(\R^3/ {\mathcal{T}}_3\big)$.  We
use the classical Lebesgue spaces $L^p=L^p(\T^3)$ and Sobolev spaces
$W^{k, p}=W^{k, p}(\T^3)$, with $H^k:=W^{k, 2}$, for $p, k \in \N$,
in the periodic setting.  We denote by $\|\cdot\|$ the
$L^2(\T^3)$-norm and the associated operator norms and we always
impose the zero mean condition on the considered fields. In the
sequel, we will use the same notation for scalar and vector-valued
functions, since no ambiguity occurs.  Moreover, dealing with
divergence-free vector fields, we also define, for a general exponent
$s \geq 0$, the following spaces
\begin{equation*}
  H_s :=
  \Big\{ v : \T^3 \to \R^3\, \colon\,  v \in (H^s)^3,\,\, 
  \nabla \cdot v = 0,
  \,\,\int_{\T^3}v dx = 0\Big\}.
\end{equation*}
If $0\leq s \leq 1$, the condition $\nabla\cdot v =0$ must be
understood in the weak sense. Let $X$ be a real Banach space with norm
$\|\cdot\|_X$.  We will use the customary Bochner spaces $L^q(0, T;
X)$, with norm denoted by $\|\cdot\|_{L^q(0,T;X)}$.

For $v \in H^s$, we can expand the fields as $v(x) = \sum_{k\in
  {\mathcal{T}}_3^\star} \widehat{v}_k e^{ik\cdot x}$,
where $k \in {\mathcal{T}}_3^\star$, and the Fourier
coefficients are defined by $\widehat{v}_k = 1/|\T^3|\int_{{\T}^3}
v(x) e^{-ik \cdot x}dx$.  The magnitude of $k$ is given by $|k|^2 :=
 (k_1)^2 + (k_2)^2 +(k_3)^2$. The $H^s$ norms
are defined by $\|v\|^2_s:= \sum_{k\in
  {\mathcal{T}}_3^\star} |k|^{2s} |\widehat{v}_k|^2$, where
$\|v\|^2_0 := \|v\|^2$.  The inner products associated to these norms
are $(w, v)_{H^s} := \sum_{k\in
  {\mathcal{T}}_3^\star}|k|^{2s}\widehat{w}_k \cdot
\overline{\widehat{v}_k}$, where $\overline{\widehat{v}_k}$ denotes
the complex conjugate of $\widehat{v}_k$.  To have real valued vector
fields, we impose $\widehat{v}_{-k} = \overline{\widehat{v}_k}$ for
any $k \in {\mathcal{T}}_3^\star$ and for any field denoted
by $v$.  It can be shown (see e.g. \cite{Do-Gib-1995}) that when $s$
is an integer, $\|v\|^2_s = \|\nabla^s v\|^2$ and also, for general $s
\in \R$, $(H^s)' = H^{-s}$. All these considerations can be 
adapted straightforwardly to the case of the spaces $H_s$. 
In particular, we denote $(H_s)'$ by  $H_{-s}$.

We will denote by $C$ generic constants, which may change from line to
line, but which are independent of the diffusion coefficient
$\epsilon$, the deconvolution parameter $N$ and of the solution of the
equations we are considering.

Let us now briefly recall the properties of the
  Helmholtz filter. We also introduce some additional notation
about the the operators involved in the definition of the considered
deconvolution model.
Let $\alpha > 0$ be a given fixed number and, for
  $w \in H_s$, $s\geq -1$,  let us denote by $(\overline{w},  \pi) \in H_{s+2} \X
  H^{s+1}$, 
 the unique solution of the following
 Stokes-like problem:
  \begin{equation} \label{eq:helmholtz-filter}
    \begin{aligned}
      & \overline{w} -\alpha^2\Delta \overline{w} +
      \nabla \pi = w \textrm{ in } \T^3,\\
      & \nabla \cdot  \overline{w} = 0 \textrm{ in } \T^3,\\
      & \int_{\T^3} \overline{w} dx = 0, \,\,\, \int_{\T^3} \pi dx =
      0.
    \end{aligned}
  \end{equation}
The velocity component of $(\overline{w},  \pi)$ is denoted also
  by $\overline{w} = G_\alpha(w)$ and $A_1:=G_\alpha^{-1}$.
Let us consider an element $w \in H_s$ and take its expansion in terms of Fourier
series as $w = \sum_{k\in {\mathcal{T}}_3^\star}\widehat{w}_k
e^{ik\cdot x}$, so that inserting this expression in
\eqref{eq:helmholtz-filter}
and looking for $(\overline{\omega}, \pi)$, in terms of
Fourier
series,  we get 
\begin{equation} \label{eq:pure-utility}
\overline{w} =  \sum_{k\in  {\mathcal{T}}_3^\star}  
\frac{1}{1+\alpha^2|k|^2}
  \widehat{w}_k e^{ik\cdot x} = G_\alpha(w), \textrm{ and } \pi=0. 
\end{equation}

  For a scalar function $\chi$ we still
  denote by $\overline{\chi}$ the solution of the pure Helmholtz
  problem
  \begin{equation} \label{eq:pure-helmholtz} 
 -\alpha^2\Delta \overline{\chi} +
  \overline{\chi} = \chi\,
    \textrm{ in } \T^3,
  \end{equation}
  where $A_2\overline{\chi}:=  -\alpha^2\Delta \overline{\chi} +
  \overline{\chi}$. Further,  taking $\chi\in H^s$ the expression of
  $\overline{\chi}$ in terms of Fourier series can be
  retrieved, formally, by \eqref{eq:pure-helmholtz} substituting $\chi$
in place of $w$.

 In what follows, in order to keep the notation
compact, we use the same symbol $A$ for
the operators $A_1$ and $A_2$, distinguishing the two situations only when it is
required by the context. 
According to the above facts, the deconvolution
  operator $D_N$ in  \eqref{eq:deconv}
is actually  given by   $D_N = \sum_{n=0}^N (I -  A^{-1})^n$, $N\in\N$, 
with $A$ defined by \eqref{eq:helmholtz-filter}, 
when it is acting on the elements of $H_s$ and, by \eqref{eq:pure-helmholtz},
in the case of the scalar functions in $H^s$. 

 Notice that,, in the LES model
  \eqref{eq:Bouss-approx} and in the filtered equations
  \eqref{eq:Bouss-mean}, the symbol ``\hspace{0.05 cm}$
  \overline{\empty{{}^{{}^{\,\,\,\,}}}}$\hspace{0.05 cm}'' denotes the
  pure Helmholtz filter, applied component-wise to the various vector
  and tensor fields. Referring to the right-hand side of first equation in
  \eqref{eq:Bouss-mean}, since $e_3$ is a constant vector,
  then we have that $G_\alpha (\theta e_3) = \overline{\theta e_3} =
  \overline{\theta}e_3 = G_\alpha (\theta)e_3$ and
  $A (\overline{\theta e_3}) = A(\overline{\theta})e_3 $ (where
  the meaning of $A$ is understood in the sense stated above).

Also, for brevity, in the sequel 
we omit the explicit dependence of $G_\alpha$ on $\alpha$, and we write $G$ in place of
$G_\alpha$. 

The deconvolution operator $D_N$
 is constructed thanks to the Van~Cittert algorithm;
the  reader will find a complete description and analysis of the Van
Cittert algorithm and its variants in \cite{Lew-2009}. Here, we only
report the properties needed to describe the considered model.
Let $\omega\in H_s$
(or $\omega \in H^s$), 
%take its expansion in terms of Fourier
%series, and consider the filtered quantity $ 
%  G(\omega) =\sum_{k\in  {\mathcal{T}}_3^\star}  
%  \widehat{\omega}_k e^{ik\cdot x}/ (1+\alpha^2|k|^2).$}
starting from the expression \eqref{eq:pure-utility},
we can write the deconvolution
operator in terms of Fourier series by the formula
\begin{gather}
  \widehat{D}_N(k) = \sum_{n=0}^N \left(\frac{\alpha^2
      |k|^2}{1+\alpha^2|k|^2}\right)^n = (1 + \alpha^2 |k|^2
  )\varrho_{N,k} \label{eq:basics-DN} \intertext{where} \varrho_{N,k}
  = 1-\left( \frac{ \alpha^2|k|^2}{1 + \alpha^2|k|^2}\right)^{N+1}
  \intertext{and} D_N(\omega) = \sum_{k\in
    {\mathcal{T}}_3^\star} \widehat{D}_N (k) \widehat{\omega}_k
  e^{ik\cdot x}.
\end{gather}
The basic properties satisfied by $\widehat{D}_N$ are summarized in
the next results
\begin{lemma} For each fixed $k\in {\mathcal{T}}_3$,
  \begin{equation}
    \widehat{D}_N(k) \to 1+ \alpha^2|k|^2 = \widehat{A}_k, \textrm{ as } 
    N \to +\infty, 
  \end{equation}
  even if not uniformly in $k$.
\end{lemma}
This provides that $\{D_N\}_{N\in \N}$ converges to $A$, in some
sense, when $N \to +\infty$.  The meaning of this convergence is
specified in the next lemma (see also \cite[\S 2]{Be-Lew-2012}).

\begin{lemma} \label{lem:utility}
  For each $N\in \N$ the operator $D_N \colon H_s \to H_s$ is
  self-adjoint, it commutes with differentiation, and the following
  properties hold true:
  \begin{align} \allowdisplaybreaks &1 \leq \widehat{D}_N (k) \leq N +
    1,\,\, \forall k\in {\mathcal{T}}_3^\star,
    \label{eq:dn-1} \\
    &\widehat{D}_N (k) \cong (N + 1) \frac{1 + \alpha^2|k|^2}{\alpha^2
      |k|^2}
    \textrm{ for large } |k|, \label{eq:estimate-norm-DN}\\
    &\underset{|k|\to +\infty}{\lim}\widehat{D}_N (k) = N + 1 \textrm{
      for fixed }
    \alpha > 0, \\
    &\widehat{D}_N (k) \leq 1 + \alpha^2 |k|^2= \widehat{A}_k, \forall
    k \in {\mathcal{T}}_3^\star,
    \alpha > 0, \\
    & \textrm{the map } \omega \mapsto D_N (\omega) \textrm{ is an isomorphism
      s.t.}
    \|D_N\|_{H_s} = O(N + 1), \forall s \geq 0,\\
    & \underset{N\to +\infty}{\lim} D_N (\omega) = A\omega \textrm{ in } H_s\,\,
    \forall s \in \R \textrm{ and } \omega \in H_{s+2}.
  \end{align}
\end{lemma}
\noindent Also in this case Lemma \ref{lem:utility} can be directly extended to
the spaces $H^s$.

Finally, for the reminder of the paper we will always use the natural
notation $G = A^{-1} = (I - \alpha^2\Delta)^{-1}$.

\section{The approximate problem} \label{sec-approximate-model}
In this section we prove existence and uniqueness of suitable weak
solutions to the system \eqref{eq:Bouss-approx}.  For the reminder of
this section the parameters $N$, $\alpha$ and $\epsilon$ are fixed,
and we assume that $u_0\in H_0$, $\theta_0\in L^2(\T^3)$, which gives
$ \overline{u_0}= G(u_0)\in H_2$ as well as $ \overline{\theta_0}=
G(\theta_0) \in H^2$.  Let us recall the following definition

\begin{defin}[Regular weak solution]
  We say that a triple $(w, \rho, q)$ is a ``regular weak solution''
  to the system \eqref{eq:Bouss-approx} if the three following
  conditions are satisfied:
  \begin{align}
    \intertext{Regularity:} & w \in L^2(0, T ; H_2) \cap C([0, T
    ]; H_1),\,\,
    \rho \in L^2(0, T ; H^2) \cap C([0, T
    ]; H^1)\label{eq:w-rho-reg}\\
    & \D_tw \in L^2(0, T ; H_0), \,\, \D_t\rho \in L^2(0, T ;
    L^2(\T^3))
    \label{eq:Dt-w-rho-reg}\\
    & q \in L^2(0, T; H^1).  \intertext{Initial data:} &\underset{t\to
      0}{\lim} \| w(t, \cdot) - G(u_0)\|_{1}= 0, \,\, \underset{t\to
      0}{\lim} \| \rho (t, \cdot) - G(\theta_0)\|_{1}= 0,
    \intertext{Weak formulation: For all $(v, h)\in L^2(0, T; H_1)\X
      L^2(0, T; H^1)$} & \begin{aligned} \label{weak-form-1} \int_0^T
      \int_{\mathbb{T}_3} \D_t w \cdot v &- \int_0^T
      \int_{\mathbb{T}_3} G\big( D_N(w) \otimes
      D_N (w)\big) : \nabla v \\
      & + \nu \int_0^T \int_{\mathbb{T}_3} \nabla w : \nabla v +
      \int_0^T \int_{\mathbb{T}_3} \nabla q \cdot v =
      \int_0^T\int_{\mathbb{T}_3}\rho e_3 \cdot v,
    \end{aligned}\\[2 mm]
    & \int_0^T \int_{\mathbb{T}_3} \D_t \rho \cdot h - \int_0^T
    \int_{\mathbb{T}_3} G\big( D_N(\rho) D_N (w)\big) \cdot \nabla h +
    \epsilon \int_0^T \int_{\mathbb{T}_3} \nabla \rho \cdot \nabla h
    =0.  \label{weak-form-2}
  \end{align}
\end{defin}
Notice that, to keep the notation coincise, we suppressed all $dx$ and $dt$ from the
above space-time integrals. For the remainder of the paper we always
use this convention.

\begin{theorem} \label{preliminary-results} Assume that $u_0\in H_0$
  and that $\theta_0\in L^2(\T^3)$, with $\alpha>0$, $\epsilon >0$ and
  $N\in \N$ given. Then, problem \eqref{eq:Bouss-approx} has a unique
  regular weak solution $(w, \rho, q)$. Moreover, this solution
  satisfies the following energy equality
  \begin{equation} \label{energy-equality-2}
    \begin{aligned}
      \frac{1}{2}\frac{d}{dt} \big(\|A^{\frac{1}{2}} &
      D_N^{\frac{1}{2}}(w)\|^2 + \|\ADN (\rho)\|^2\big) +
      \nu\|\nabla  \ADN (w)\|^2\\
      & + \epsilon \|\nabla \ADN (\rho)\|^2 = \big(\ADN (\rho) e_3,
      \ADN (w) \big).
    \end{aligned}
  \end{equation}
\end{theorem}

\begin{proof}[Proof of Theorem~\ref{preliminary-results}.] 
  We follow the main lines in the proof of
  \cite[Theorem~3.1]{Be-Lew-2012} by using the Galerkin method to
  approximate a regular weak solution to the problem
  \eqref{eq:Bouss-approx} with finite dimensional velocities and
  densities.  We now proceed with the following steps.
 
  \indent$\mathbf{Step\,\, 1}$: \emph{Construction of the
    approximations for velocity and density.}  Since the construction
  of the approximate solutions is very classical, we will only sketch
  it (for more details see, e.g., \cite{Guo-1995, Lions}).  Let be
  given $m \in \N\backslash\{0\}$ and define
  \begin{align*}
    &V^m = \Big\{ w \in H^1 \, \colon \, \int_{\T^3} w(x) e^{-ik\cdot
      x}dx = 0,\,\, \forall k
    \textrm{ with } |k| > m \Big\},\\
    &\mathbf{V}_m = \Big\{ \mathbf{w} \in H_1 \, \colon \, \int_{\T^3}
    \mathbf{w}(x) e^{-ik\cdot x}dx = 0,\,\, \forall k \textrm{ with }
    |k| > m \Big\},
  \end{align*}
  and let $\{E_j\}_{j=1,\ldots, d_m}$ and
  $\{\mathbf{E}_j\}_{j=1,\ldots, \delta_m}$ be orthogonal bases of
  $V^m$ and $\mathbf{V}_m$ respectively.  Without loss of generality,
  we can assume that the $E_j$'s are
 eigen-functions of the operator $I -\alpha^2\Delta$ 
introduced in  \eqref{eq:pure-helmholtz}  as well as the
$\mathbf{E}_j$'s
 are eigen-functions of
  the Stokes-like operator associated to \eqref{eq:helmholtz-filter}.
Further, the $E_j$'s and $\mathbf{E}_j$'s are supposed to have unitary norm.
 We denote by $P_m$ the orthogonal
  projection from $H^1$ onto $V^m$ and, similarly, by $\mathbf{P}_m$ the the
  orthogonal projection from $H_1$ onto $\mathbf{V}_m$.
  
  For every positive integer $m$, we look for an approximate solution
  of problem \eqref{eq:Bouss-approx} of the form
  \begin{equation*}
    \rho_m (t, x) = \sum^{d_m}_{j=1}\rho_{m, j}(t) E_j (x)\, \textrm{ and }\, 
    w_m(t, x) =  \sum^{\delta_m}_{j=1}w_{m, j}(t) \mathbf{E}_j(x).
  \end{equation*}
  Thanks to the Cauchy-Lipschitz Theorem, we can prove the existence
  of a unique $C^1$ maximal solution $\big(w_m(t), \rho_m(t)\big)\in
  \mathbf{V}_m\X V^m$ for all $t \in [0, T_m)$ where $T_m>0$ is the
  maximal existence time, to the system
  \begin{align}
    &\begin{aligned} \label{weak-form-m-1} \int_{\mathbb{T}_3} \D_t
      w_m \cdot v &- \int_{\mathbb{T}_3} G\big( D_N(w_m) \otimes
      D_N (w_m)\big) : \nabla v \\
      & + \nu \int_{\mathbb{T}_3} \nabla w_m : \nabla v =
      \int_{\mathbb{T}_3}\rho_m e_3 \cdot v,
    \end{aligned}\\[2 mm]
    & \int_{\mathbb{T}_3} \D_t \rho_m \cdot h - \int_{\mathbb{T}_3}
    G\big( D_N(\rho_m) D_N (w_m)\big) \cdot \nabla h + \epsilon
    \int_{\mathbb{T}_3} \nabla \rho_m \cdot \nabla h
    =0, \label{weak-form-m-2}
  \end{align}
  for all $(v, h)\in L^2(0, T_m; \mathbf{V}_m)\X L^2(0, T_m; V^m)$. As
  we will see in the sequel, we can actually take $T_m = T$, and this
  concludes the construction of the approximate solutions $\{(w_m,
  \rho_m)\}_{m\in \N}$.

  \indent$\mathbf{Step\,\, 2}$: \emph{Energy a priori estimates} The
  natural and correct test functions in
  \eqref{weak-form-m-1}-\eqref{weak-form-m-2} to get a priori
  estimates are $\big(AD_N(w_m), AD_N(\rho_m)\big)$.  Since $A$ is
  self-adjoint and commute with the differential operators, it holds
  that
  \begin{equation*}
    \begin{aligned}
      & \int_{\mathbb{T}_3} G\big( D_N(w_m) \otimes D_N (w_m)\big) :
      \nabla  (AD_N(w_m)) dx=0, \\
      & \int_{\mathbb{T}_3} G\big( D_N(\rho_m) D_N (w_m)\big) \cdot
      \nabla (AD_N(\rho_m)) dx =0.
    \end{aligned}
  \end{equation*}
Moreover, since $A$ and $D_N$ commute, the following identities hold
  true
%  \begin{subequations}
%    \allowdisplaybreaks
  \begin{align*}
    &\big(\D_t w_m, AD_N (w_m)\big) = \frac{1}{2}\frac{d}{dt}
    \|A^{\frac{1}{2}} D^{\frac{1}{2}}_N (w_m)\|^2,\\
    &\big(-\Delta w_m, AD_N (w_m)\big) = \|\nabla A^{\frac{1}{2}}
    D^{\frac{1}{2}}_N (w_m)\|^2,\\
    &\big(\D_t \rho_m, AD_N (\rho_m) \big) =
    \frac{1}{2}\frac{d}{dt}\|A^{\frac{1}{2}}
    D^{\frac{1}{2}}_N (\rho_m) \|^2,\\
    &\big(-\Delta \rho_m, AD_N (\rho_m) \big) = \|\nabla
    A^{\frac{1}{2}} D^{\frac{1}{2}}_N (\rho_m) \|^2,\\
    & \begin{aligned} 
      \big( \rho_m e_3, AD_N (w_m) \big) = &\big(
      A^{\frac{1}{2}}D^{\frac{1}{2}}_N \big(\rho_m e_3\big),
      A^{\frac{1}{2}}D^{\frac{1}{2}}_N (w_m) \big)\\
    =&\big(  A^{\frac{1}{2}}D^{\frac{1}{2}}_N (\rho_m) e_3,
    A^{\frac{1}{2}}D^{\frac{1}{2}}_N (w_m) \big),
  \end{aligned}
\end{align*}
% \end{subequations}
where the meaning of the operators $A$ and $D_N$, 
  as previously discussed,  depends on the type  of their arguments.

  Therefore, with usual computations, we obtain
  \begin{equation} \label{energy-equality} \left\{ \begin{array}{l}
        \begin{aligned}
          \frac{1}{2}\frac{d}{dt}\|\ADN (w_m)\|^2 +
          \nu\|\nabla \ADN (&w_m) \|^2 \\
          & = \big(\ADN (\rho_m) e_3, \ADN (w_m) \big),
        \end{aligned}\\[1.5 em]
        \displaystyle \frac{1}{2}\frac{d}{dt}\|\ADN (\rho_m)\|^2 +
        \epsilon \|\nabla \ADN (\rho_m)\|^2 = 0,
      \end{array}\right.
  \end{equation}
  which provide the following energy equality
  \begin{equation} \label{energy-equality-1}
    \begin{aligned}
      \frac{1}{2}\frac{d}{dt} \Big(\|A^{\frac{1}{2}}
      D_N^{\frac{1}{2}}(w_m)\|^2 &+ \|\ADN (\rho_m)\|^2\Big) +
      \nu\|\nabla  \ADN (w_m)\|^2\\
      & + \epsilon \|\nabla \ADN (\rho_m)\|^2 = \big(\ADN (\rho_m)
      e_3, \ADN (w_m) \big).
    \end{aligned}
  \end{equation}
  Again, from relations in \eqref{energy-equality}, integrating by
  parts and using Poincar\'e's inequality together with Young's
  inequality, we get
  \begin{equation} \label{eq:a-priori-estimates}
    \left\{ \begin{array}{l}
        \begin{aligned}
          \|\ADN (w_m)\|^2 +
          \nu\int_0^t \|& \nabla  \ADN  (w_m)\|^2\\[- 0.5 em]
          &\leq \|\ADN \mathbb{P}_m(\overline{u_0})\|^2 +
          \frac{1}{\nu}\int_0^t \| A^{\frac{1}{2}} D^{\frac{1}{2}}_N
          (\rho_m) \|^2,
        \end{aligned}\\
        \|\ADN (\rho_m)\|^2 + 2\epsilon \displaystyle \int_0^t\|\nabla
        \ADN (\rho_m)\|^2 \leq \|\ADN \mathbb{P}_m
        (\overline{\theta_0})\|^2,
      \end{array}\right.
  \end{equation}
  with $t\in [0, T_m)$. Recalling the properties of $D_N$, that $P_m$
  and $\mathbf{P}_m$ and commute with $A$ and $D_N$ and that the
  operator $A^{-\frac{1}{2}}D_N^{\frac{1}{2}}$ has for symbol
  $0\leq\varrho^{1/2}_{N,j} \leq 1$, we get
  \begin{equation*}
    \|\ADN \mathbb{P}_m(\overline{u_0})\|
    =\|\mathbb{P}_m\ADN (\overline{u_0})\|\leq
    \|\ADN(\overline{u_0})\|= \|A^{-\frac{1}{2}}D_N^{\frac{1}{2}}(u_0)\|
    \leq \|u_0\|,
  \end{equation*}
  as well as $\|\ADN \mathbb{P}_m(\overline{\theta_0})\|\leq
  \|\theta_0\|$. Hence, adding the two inequalities in
  \eqref{eq:a-priori-estimates} we obtain that, for $t\in [0, T_m)$
  \begin{equation}
    \begin{aligned} \label{eq:first-bound}
      \|\ADN (w_m)\|^2  &+
      \|\ADN (\rho_m)\|^2 +
      \nu\int_0^t\|\nabla A^{\frac{1}{2}}D_N^{\frac{1}{2}}(w_m)\|^2\\
      &+ 2\epsilon\int_0^t\|\nabla \ADN (\rho_m)\|^2 \leq 
      \|u_0\|^2 +  \big(1 + \frac{t}{\nu}\big)  \|\theta_0\|^2.
    \end{aligned}
  \end{equation}
  Exploiting the fact that $E_j$'s and $\mathbf{E}_j$'s are
  eigen-functions for both $A$ and $D_N$ and hence also for $\ADN$, and
  recalling \eqref{eq:basics-DN}, then we also have that
  \begin{equation} \label{eq:bound-for-wm-rhom}
    \sum^{d_m}_{j=1}\varrho_{N,j}w_{m,j}^2(t) +
    \sum^{\delta_m}_{j=1}\varrho_{N,j}\rho_{m,j}^2(t) \leq C\big(
    \|u_0\|^2 + \big(1 +\frac{t}{\nu}\big) \|\theta_0\|^2 \big).
  \end{equation}
  In particular, from \eqref{eq:first-bound} and 
\eqref{eq:bound-for-wm-rhom}, we infer that
  the maximal solution $(w_m, \rho_m)$ of
  \eqref{weak-form-m-1}-\eqref{weak-form-m-2} is global. Otherwise, if
  $T_m$ is finite, then the right-hand side of \eqref{eq:first-bound}
  will be finite as well and the approximate solution $(w_m, \rho_m)$
  (that can not blow-up in $[0, T_m]$ being bounded) would have a life
  span strictly larger than $T_m$, which is in contradiction with
  maximality.  Therefore, we can take $T_m = T$ for any $T <+\infty$,
  and the approximate solutions are well-defined on $[0, +\infty)$.

  \indent$\mathbf{Step\,\, 3}$: \emph{Further a priori estimates.}  We
  now provide other suitable bounds, for the terms $w_m$ and $\rho_m$,
  that are summarized in the tables below.  Let us consider the first
  table \eqref{table}, the other one is organized in a similar way.
  In the first column we have labeled the estimates. The second
  specify the quantity that is bounded in the norm indicated in column
  third.  Finally, fourth column states the order of magnitude of the
  norms in terms of the relevant parameters $\alpha$, $\epsilon$, $m$
  and $N$. Notice that $O(1)=O(1; m)$ means a uniform bound with
  respect to $m$.  For instance, the meaning of
  ($\eqref{table}$-$(a)$) is that $A^{\frac{1}{2}}D_N^{\frac{1}{2}}
  (w_m)$ is bounded in $L^\infty(I; H_0)\cap L^2(I; H_1)$, where
  $I=[0,T]$, $T>0$, with order $O(1; m)$.  All the estimates, except
  (\eqref{table}-(g)) and (\eqref{table-rho}-(g)), are also uniform
  with respect to the deconvolution parameter $N$.
  \begin{equation} \label{table} {
      \begin{tabular}{|@{\,}l@{\,}|@{\,}l@{\,}|@{\,}l@{\,}|@{\,}l@{\,}|}
        \hline
        Label   &   Variable & Bound   &    Order \\
        \hline \hline 
        (a)    & $A^{\frac{1}{2}}D_N^{\frac{1}{2}} (w_m)$ 
        &  $L^\infty(I; H_0)\cap L^2(I; H_1)$  
        &   $O(1)$ \\
        \hline       
        (b)    & $D_N^{\frac{1}{2}} (w_m)$ 
        &  $L^{\infty}(I; H_0)\cap L^2(I; H_1)$  
        &   $O(1)$  \\
        \hline
        (c)    & $D_N^{\frac{1}{2}} (w_m)$ 
        &   $L^{\infty}(I; H_1)\cap L^2(I; H_2)$  
        &   $O(\alpha^{-1})$ \\
        \hline
        (d)    & $ w_m$ 
        &   $L^{\infty}(I; H_0)\cap L^2(I; H_1)$  
        &   $O(1)$ \\
        \hline
        (e)    & $ w_m$
        &   $L^{\infty}(I; H_1)\cap L^2(I; H_2)$  
        &   $O(\alpha^{-1})$  \\
        \hline
        (f)    & $D_N (w_m)$
        &   $L^{\infty}(I; H_0)\cap L^2(I; H_1)$  
        &   $O(1)$ \\
        \hline
        (g)    & $D_N (w_m)$
        &   $L^{\infty}(I; H_1)\cap L^2(I; H_2)$  
        &   $O(\frac{\sqrt{N+1}}{\alpha})$ \\
        \hline
        (h)    & $\D_t  w_m$
        &  $L^2(I; H_0)$  
        &  $O(\alpha^{-1})$\\ 
        \hline  
      \end{tabular}
    }
  \end{equation}
  and
  \begin{equation} \label{table-rho} {
      \begin{tabular}{|@{\,}l@{\,}|@{\,}l@{\,}|@{\,}l@{\,}|@{\,}l@{\,}|@{\,}l@{\,}|
          @{\,}l@{\,}|}
        \hline
        Label   &   Variable & Bound-A  &  Order-Bound-A & Bound-B   & 
        Order-Bound-B  \\
        \hline \hline 
        (a)    & $A^{\frac{1}{2}}D_N^{\frac{1}{2}} (\rho_m)$ 
        &   $L^\infty(I; L^2)$
        &   $O(1)$
        &   $L^2(I; H^1)$  
        &   $O({\sqrt{\epsilon}}^{-1}) $    \\
        \hline       
        (b)    & $D_N^{\frac{1}{2}} (\rho_m)$ 
        &   $L^{\infty}(I; L^2)$ 
        &   $O(1)$ 
        &   $L^2(I; H^1)$  
        &  $O(\sqrt{\epsilon}^{-1})$   \\
        \hline
        (c)    & $D_N^{\frac{1}{2}} (\rho_m)$ 
        &  $L^{\infty}(I; H^1)$
        &  $O(\alpha^{-1})$ 
        &  $L^2(I; H^2)$  
        &  $O((\alpha\sqrt{\epsilon})^{-1})$   \\
        \hline
        (d)    & $ \rho_m$ 
        &   $L^{\infty}(I; L^2)$
        &   $O(1)$ 
        &   $L^2(I; H^1)$  
        & $O(\sqrt{\epsilon}^{-1})$   \\
        \hline
        (e)    & $ \rho_m$
        &   $L^{\infty}(I; H^1)$   
        &   $O(\alpha^{-1})$
        &   $L^2(I; H^2)$  
        & $O((\alpha\sqrt{\epsilon})^{-1})$    \\
        \hline
        (f)    & $D_N (\rho_m)$
        &   $L^{\infty}(I; L^2)$  
        &   $O(1)$
        &   $L^2(I; H^1)$          
        & $O(\sqrt{\epsilon}^{-1}) $    \\
        \hline
        (g)    & $D_N (\rho_m)$
        &   $L^{\infty}(I; H^1)$ 
        &   $O(\frac{\sqrt{N+1}}{\alpha})$
        &   $L^2(I; H^2)$  
        &   $O(\frac{\sqrt{N+1}}{\alpha\sqrt{\epsilon}})$   \\
        \hline
        (h)    & $\D_t  \rho_m$
        &   $L^2(I; L^2)$  
        &  $O(\alpha^{-1})$
        &
        & \\ 
        \hline  
      \end{tabular}
    }
  \end{equation}
  \indent $\bullet\,\, Checking$ (\eqref{table}-(a)) $and$
  (\eqref{table-rho}-(a)): These bounds are a straightforward
  consequence of inequalty \eqref{eq:first-bound}.  In particular,
  (\eqref{table-rho}-(a)) means that $\ADN (\rho_m) \in L^\infty(I,
  L^2(\T^3))$ and that $\sqrt{\epsilon}\nabla \ADN (\rho_m) \in L^2(I,
  L^2(\T^3))$ both with order $O(1)$.

  \indent $\bullet\,\, Checking$
  (\eqref{table}-(b))-(\eqref{table}-(c)) $and$
  (\eqref{table-rho}-(b))-(\eqref{table-rho}-(c)): Let $v\in
  H_2$. Then, with obvious notations one has
  \begin{equation*}
    \|A^{\frac{1}{2}}v\|^2 = \sum_{k\in  {\mathcal{T}}^\star_3}
    (1+ \alpha^2 |k|^2) |\widehat{v}_k|^2 =\|v\|^2 + \alpha^2 \|\nabla v\|^2.
  \end{equation*}
  It suffices to apply this identity to $v = D_N^{\frac{1}{2}}(w_m)$
  and to $v = \D_iD_N^{\frac{1}{2}}(w_m)$, $i = 1, 2, 3$, in
  \eqref{eq:first-bound} to get the claimed results. The same
   considerations can be used for $\rho_m$ to prove  (\eqref{table-rho}-(b))
    and (\eqref{table-rho}-(c)).

  \indent $\bullet\,\, Checking$
  (\eqref{table}-(d))-(\eqref{table}-(e)) $and$
  (\eqref{table-rho}-(d))-(\eqref{table-rho}-(e)): These bounds are
  consequence of (\eqref{table}-(b))-(\eqref{table}-(c))
  (respectively, \eqref{table-rho}-(b))-(\eqref{table-rho}-(c)))
  combined with \eqref{eq:dn-1}, which give
  \begin{equation*}
    \|v\|_s \leq \|D_N(v)\|_s \leq (N +1)\|v
    \|_s,
  \end{equation*}
  for general $v$ and for any $s \geq 0$. In particular, for the case
  of (\eqref{table-rho}-(d)), we obtain that $\rho_m \in L^\infty(I,
  L^2(\T^3))$ and that $\sqrt{\epsilon}\nabla \rho_m \in L^2(I,
  L^2(\T^3))$ both with order $O(1)$. Similarly, for the case of
  (\eqref{table-rho}-(e)), it follows that $\rho_m \in L^\infty(I,
  H^1)$ and that $\sqrt{\epsilon}\nabla \rho_m \in L^2(I, H^1)$ both
  with order $O(\alpha^{-1})$.

  \indent $\bullet\,\, Checking$ (\eqref{table}-(f)) $and$
  (\eqref{table-rho}-(f)): The operator
  $A^{\frac{1}{2}}D^{\frac{1}{2}}_N$ has for symbol $(1 +
  \alpha^2|k|^2)\varrho^{\frac{1}{2}}_{N,k}$ while the the one of
  $D_N$ is $(1 + \alpha^2|k|^2)\varrho_{N,k}$.  Since $0\leq
  \varrho_{N,k} \leq 1$, then $\|D_N(v)\|_s \leq \|A^{\frac{1}{2}}
  D^{\frac{1}{2}}_N (v)\|_s$, for general $v$ and for any $s \geq 0$.
  Hence, the bounds follow as a consequence of (\eqref{table}-(a)) and
  (\eqref{table-rho}-(a)).
 
  \indent $\bullet\,\, Checking$ (\eqref{table}-(g)) $and$
  (\eqref{table-rho}-(g)): These relations follow directly from
  (\eqref{table}-(e)) (respectively, (\eqref{table-rho}-(e))) used
  together with \eqref{eq:dn-1}.

  \indent $\bullet\,\, Checking$ (\eqref{table}-(h)) $and$
  (\eqref{table-rho}-(h)):
  Using $\D_tw_m\in \mathbf{V}_m$ and $\D_t\rho_m \in V^m$ as test
  functions in the equations \eqref{weak-form-m-1} and
  \eqref{weak-form-m-2} respectively, we get
  \begin{equation} \label{eq:stime-derivate-m}
    \left\{ \begin{array}{l} \displaystyle \| \D_t w_m\|^2 +
        \int_{\mathbb{T}^3} A_{N, m} \cdot \D_t w_m + \frac{\nu}{2}
        \frac{d}{dt}\| \nabla w_m \|^2
        =   \int_{\mathbb{T}_3}\rho_m e_3 \cdot \D_t w_m,\\[1.5 em]
        \displaystyle \| \D_t \rho_m \|^2 + \int_{\mathbb{T}_3} B_{N,
          m} \cdot \D_t \rho_m + \frac{\epsilon}{2} \frac{d}{dt} \|
        \nabla \rho_m\|^2=0,
      \end{array}\right.
  \end{equation}
  where
  \begin{align*}
    &A_{N, m}:=   G\big( \nabla \cdot [D_N(w_m) \otimes D_N (w_m)]\big),\\
    &B_{N, m}:= G\big( \nabla \cdot [D_N(\rho_m) D_N (w_m)]\big).
  \end{align*}
  To estimate the time derivatives, we need to bound the the bi-linear
  terms $A_{N, m}$ and $B_{N, m}$. The former can be treated as done
  in \cite{Be-Lew-2012}. In fact, observing that $D_N (w_m) \in L^4(0,
  T; L^3(\mathbb{T}^3)^3)$ with order $O(1)$, we obtain that
  $D_N(w_m)\otimes D_N(w_m) \in L^2(0,T;
  L^{\frac{3}{2}}(\mathbb{T}^3)^9)$ with order $O(1)$. Further, we also
  have that $D_N(\rho_m)D_N(w_m) \in L^2(0,T;
  L^{\frac{3}{2}}(\mathbb{T}^3)^3)$ with order $O(1)$.  Indeed, by
  applying H\"older's inequality and the embedding $H^1\hookrightarrow
  L^6(\T^3)$, we get
  \begin{align*}
    \int_0^T\bigg[\int_{\T^3}|D_N (\rho_m)|^{\frac{3}{2}} |D_N
    (w_m)|^{\frac{3}{2}}\bigg]^{\frac{4}{3}} &\leq \int_{0}^T\|D_N
    (\rho_m)\|^{2}
    \|D_N (w_m)\|^{2}_{L^6(\T^3)}\\
    &\leq \|D_N (\rho_m)\|^2_{L^\infty(0,T; L^2(\T^3))} \|D_N
    (w_m)\|^2_{L^2(0,T; H^1)},
  \end{align*}
  and the conclusion follows recalling (\eqref{table}-(d)) and
  (\eqref{table-rho}-(f)).
  
  Since the operator $(\nabla\cdot ) \circ \, G$ has symbol
  corresponding to the inverse of one space derivative, and its norm
  is of order $O(\alpha^{-1})$, it follows that $A_{N, m}\in L^2(0, T;
  W^{1,\frac{3}{2}}(\mathbb{T}^3)^3)$ as well as $B_{N, m}\in L^2(0, T;
  W^{1,\frac{3}{2}}(\mathbb{T}^3))$ both with order
  $O(\alpha^{-1})$. Moreover, these bound yield $A_{N,m}\in L^2([0,
  T]\X \T^3)^3$ as well as $B_{N, m} \in L^2([0, T]\X \T^3)$ since $W^{1,3/2}(\T^3) \hookrightarrow
  L^2(\T^3)$ and the space $L^2(0, T; L^2(\T^3))$ is isomorphic to
  $L^2([0,T ]\X \T^3)$.  Thus, from \eqref{eq:stime-derivate-m}, with
  standard calculations we obtain
  \begin{align*}
    \|\D_t w_m \|^2 + \|\D_t\rho_m &\|^2 + \frac{\nu}{2}
    \frac{d}{dt}\| \nabla w_m \|^2 +
    \frac{\epsilon}{2} \frac{d}{dt}\| \nabla \rho_m \|^2 \\
    &\leq \| \rho_m \|^2 + \frac{1}{4} \|\D_t w_m \|^2 +
    \int_{\mathbb{T}^3} |A_{N, m} \cdot \D_t w_m|
    + \int_{\mathbb{T}^3} |B_{N, m} \cdot  \D_t \rho_m|\\
    &\leq \| \rho_m \|^2 + \frac{1}{2} \|\D_t w_m \|^2 +
    \frac{1}{2}\|\D_t \rho_m|^2 + \|A_{N, m}\|^2 + \frac{1}{2}\|B_{N,
      m}\|^2, 
  \end{align*}
  and hence \vspace{-0.08 cm}
  \begin{align*}
    \|\D_t w_m \|^2 + \|\D_t\rho_m\|^2 + \nu \frac{d}{dt}\| \nabla w_m
    \|^2 +&
    \epsilon \frac{d}{dt}\|  \nabla \rho_m \|^2 \\
    &\leq 2\| \rho_m \|^2 + 2\|A_{N, m}\|^2 + \|B_{N, m}\|^2.
  \end{align*}
  Since $\|\nabla w_m(0)\|= \|\nabla \mathbf{P}_m(\overline{u_0})\|
  =\|\mathbf{P}_m\nabla (\overline{u_0})\|\leq C\alpha^{-1} \|u_0\|$
  as well as $\|\nabla \rho_m(0)\| = \|P_m\nabla
  (\overline{\theta_0})\| \leq C\alpha^{-1} \|\theta_0\|$, then
  recalling (\eqref{table}-(e)) and (\eqref{table-rho}-(e)), the
  bounds in (\eqref{table}-(h)) and (\eqref{table-rho}-(h)) follow
  easily.

  \indent $\mathbf{Step\,\, 4}$: \emph{Taking the limit in the
    equations for} $m\to+\infty$, \emph{when $N$ and $\epsilon$ are
    fixed}.  Thanks to the bounds in \eqref{table}, we can extract
  from $\{(w_m, \rho_m)\}_{m\in \N}$ a sub-sequence (still denoted
  $\{(w_m, \rho_m)\}_{m\in \N}$) which converge to $(w, \rho)$ such
  that $w\in L^\infty(0,T ;H_1) \cap L^2(0,T; H_2)$ and $\rho\in
  L^\infty(0,T; H^1) \cap L^2(0,T; H^2)$.  Using Aubin-Lions theorem,
  by ((\eqref{table})-(d)) and ((\eqref{table})-(h)), we get
%  \begin{align}
%    &\left. \begin{array}{ll} \left.\begin{array}{l}
%            w_m \to w \\
%            \rho_m \to \rho
%          \end{array}\right\}
%        \textrm{ weakly in } L^2(0, T; H^2),
%      \end{array}\right. \label{eq:convergence-weak-w-rho}
%    \\
%    &\left. \begin{array}{ll} \left.\begin{array}{l}
%            w_m \to w\\
%            \rho_m \to \rho
%          \end{array}\right\}
%        \textrm{ strongly in } L^p(0, T ; H^1),
%        \forall p \in [1,\infty[,
%      \end{array}\right. \label{eq:convergence-strong-w-rho}
%    \\
%    &\left. \begin{array}{ll} \left.\begin{array}{l}
%            \D_tw_m \to \D_tw \\
%            \D_t\rho_m \to \D_t\rho
%          \end{array}\right\}
%        \textrm{ weakly in } L^2(0, T;
%        L^2(\T^3)).
%      \end{array}\right. \label{eq:convergence-weak-q}
%  \end{align}
  \begin{align}
    &\left. \begin{array}{ll} 
        w_m \to w  &
        \left\{ \begin{array}{l}        \textrm{ weakly in } L^2(0, T; H_2),\\
            \textrm{ strongly in } L^p(0, T ; H_1), \forall p \in [1,\infty[,
          \end{array} \right.
      \end{array}\right. \label{eq:convergence-weak-w-rho}
    \\
    &\left. \begin{array}{ll} 
            \rho_m \to \rho &
         \left\{ \begin{array}{l} 
  \textrm{ weakly in } L^2(0, T; H^2),\\
        \textrm{ strongly in } L^p(0, T ; H^1),
        \forall p \in [1,\infty[,
      \end{array}\right.
    \end{array}\right. \label{eq:convergence-strong-w-rho}\\
    &            \D_tw_m \to \D_tw   \textrm{ weakly in } L^2(0, T; H_0), \\
   &         \D_t\rho_m \to \D_t\rho    \textrm{ weakly in } L^2(0, T;
        L^2(\T^3)).
     \end{align}
  As a direct consequence of these convergences types, it follows that
  $(w, \rho)$ satisfies \eqref{eq:w-rho-reg}-\eqref{eq:Dt-w-rho-reg}.
  From \eqref{eq:convergence-weak-w-rho}-\eqref{eq:convergence-strong-w-rho}
  and the continuity of $D_N$
  in $H^s$ and $H_s$, we have that $D_N (w_m)$ and  $D_N (\rho_m)$ 
are strongly convergent,  respectively, to $D_N (w)$ and $D_N (\rho)$ in $L^4([0, T ] \X
  \T^3)$. Hence, the corresponding bi-linear terms $A_{N, m}$
  and $B_{N, m}$ converge strongly, respectively, to $D_N
  (w)\otimes D_N(w)$ and $D_N (\rho)D_N(w)$ in $L^2([0, T ] \X \T^3)$.
  This is sufficient to pass to the limit in the weak formulation
  \eqref{weak-form-m-1}-\eqref{weak-form-m-2} (see \cite{Be-Lew-2012})
  and to get that: For all $(v, h)\in L^2(0, T; H_1)\X L^2(0, T; H^1)$
  \begin{align*}
    & \begin{aligned} \int_0^T \int_{\mathbb{T}_3} \D_t w & \cdot v -
      \int_0^T \int_{\mathbb{T}_3} G\big( D_N(w) \otimes
      D_N (w)\big) : \nabla v \\
      & + \nu \int_0^T \int_{\mathbb{T}_3} \nabla w : \nabla v +
      \int_0^T \int_{\mathbb{T}_3} \nabla q \cdot v =
      \int_0^T\int_{\mathbb{T}_3}\rho e_3 \cdot v,
    \end{aligned}\\[2 mm]
    & \int_0^T \int_{\mathbb{T}_3} \D_t \rho \cdot h - \int_0^T
    \int_{\mathbb{T}_3} G\big( D_N(\rho) D_N (w)\big) \cdot \nabla h +
    \epsilon \int_0^T \int_{\mathbb{T}_3} \nabla \rho \cdot \nabla h
    =0.
  \end{align*}

  Now, in order to recover the pressure, we take the divergence of the
  equation for $w$ in \eqref{eq:Bouss-approx}, to get
  \begin{equation} \label{DeRham-1} \Delta q = \nabla \cdot ( \rho e_3
    + A_N ) ,
  \end{equation}
  where
  \begin{equation*}
    A_N := −G\big( \nabla \cdot [D_N (w)
    \otimes D_N (w)]\big).
  \end{equation*}
  Moreover, $\rho \in L^2(0, T; H^1)$ (and much more) and, due to the
  regularity of $w$, we have that $A_N \in L^2([0,T] \X \T^3)^3$ as well
  as $\Delta w \in L^2([0,T ] \X \T^3)^3$.  Then, the elements $v \in
  L^2(0,T; H_0)$ are admissible test fields for the weak formulation
  for $w$, given by \eqref{weak-form-1}, that we rewrite as follows:
  For all $v \in L^2(0,T; H_0)$, it holds true that
  \begin{equation} \label{DeRham-2}
    \int^T_0\int_{\T^3}\big(\D_t w + A_N -
    \nu \Delta w - \rho e_3\big) \cdot v dx ds = 0,
  \end{equation}
  and in particular the vector field in parentheses is orthogonal to
  divergence-free vector fields in $L^2(\T^3)$.  Hence, using
  \eqref{DeRham-1} together with \eqref{DeRham-2} and the regularity
  of $A_N$, by De Rham's Theorem one retrieves the pressure field $q
  \in L^2(0, T ; H^1)$.  \smallskip

  Lastly, we have that the energy inequality \eqref{energy-equality-2}
  holds true.  Indeed, due to the obtained regularity for $(w, \rho,
  q)$, we can use $\big(AD_N(w), AD_N(\rho)\big)$ as a test in the
  weak formulation \eqref{weak-form-1}-\eqref{weak-form-2}.
  Therefore, it follows easily that $\big(\ADN (w), \ADN (\rho)\big)$
  satisfies \eqref{energy-equality-2}.

  \indent $\mathbf{Step\,\, 5}$: \emph{Uniqueness.}  Let $(w_1,
  \rho_1, q_1)$ and $(w_2, \rho_2, q_2)$ be two solutions of
  \eqref{eq:Bouss-approx} and consider the differences $W := w_1 -
  w_2$ and $\Gamma := \rho_1 -\rho_2$. We will use $AD_N(W)$ and
  $AD_N(\Gamma)$ as test functions for the equations satisfied by $W$
  and $\Gamma$ respectively. Taking the inner product of the equation
  for $W$ against $A D_N (W)$ and integrating by parts, we get
  \begin{subequations}
    \begin{align*}
      &\begin{aligned} \frac{1}{2} &\frac{d}{dt} \| \ADN (W)\|^2
        + \nu \|\nabla \ADN (W)\|^2\\
        & \leq \int_{\mathbb{T}^3}|\nabla D_N (w_1)| |D_N (W)|^2 +
        \int_{\mathbb{T}^3} | \ADN (\Gamma)|
        |\ADN (W)|\\
        & \leq \|D_N^{\frac{1}{2}}(W)\|_{L^4}^2 \|\nabla D_N (w_1)\| +
        \frac{1}{2}\| \ADN (W)\|^2
        + \frac{1}{2} \| \ADN  (\Gamma)\|^2\\
        &\leq \|D_N(W)\|^{\frac{1}{2}} \|\nabla D_N
        (W)\|^{\frac{3}{2}} \|\nabla D_N (w_1)\| + \frac{1}{2} \| \ADN
        (W)\|^2+ \frac{1}{2} \| \ADN (\Gamma)\|^2,
      \end{aligned}
      \intertext{Similarly, taking the product of the equation for
        $\Gamma$ against $AD_N(\Gamma)$, we get} &\begin{aligned}
        \frac{1}{2}\frac{d}{dt} \| \ADN &\Gamma\|^2 + \epsilon
        \|\nabla
        \ADN \Gamma\|^2\\
        &\leq \int_{\mathbb{T}^3} |D_N(W)|\,
        |\nabla D_N (\rho_1)|\, | D_N(\Gamma)| \\
        &\leq \|D_N(\Gamma)\|_{L^4}
        \|D_N(W)\|_{L^4}\|\nabla D_N(\rho_1)\|\\
        &\leq \| D_N(\Gamma)\|^{\frac{1}{4}} \|\nabla D_N
        (\Gamma)\|^{\frac{3}{4}} \| D_N(W)\|^{\frac{1}{4}} \|\nabla
        D_N (W)\|^{\frac{3}{4}}\|\nabla D_N (\rho_1)\|.
      \end{aligned}
    \end{align*}
  \end{subequations}
  Now, adding the two above inequalities, and recalling that
  \begin{subequations}
    \begin{align*}
      &\|D_N(w)\|\leq \|A^{\frac{1}{2}}D^{\frac{1}{2}}_N(w)\|,\,\,
      \|\nabla D_N(w)\|\leq \|\nabla A^{\frac{1}{2}}D^{\frac{1}{2}}_N(w)\|,\\
      &\|D_N(\rho)\|\leq
      \|A^{\frac{1}{2}}D^{\frac{1}{2}}_N(\rho)\|,\,\, \|\nabla
      D_N(\rho)\|\leq \|\nabla
      A^{\frac{1}{2}}D^{\frac{1}{2}}_N(\rho)\|,
    \end{align*}
  \end{subequations}
  that $D_N$ and $\nabla$ commute and that $\|D_N\| = (N + 1)$, we
  then obtain
  \begin{align*}
    & \frac{1}{2}\frac{d}{dt}\Big(\| \ADN (W)\|^2 + \| \ADN
    \Gamma\|^2\Big) + \frac{\nu}{2} \|\nabla \ADN (W)\|^2 +
    \frac{\epsilon}{2} \|\nabla
    \ADN(\Gamma)\|^2 \\
    &\leq C(N + 1)^4 \Big(\!\sup_{t\geq 0} \|\nabla w_1\|^4\!\Big)
    \Big[ \frac{1}{\nu^3} \|A^{\frac{1}{2}}
    D_N^{\frac{1}{2}}(W)\|^2\Big] \!+\!  \frac{1}{2}\|A^{\frac{1}{2}}
    D_N^{\frac{1}{2}}(W)\|^2 \!+\!  \frac{1}{2}\|A^{\frac{1}{2}}
    D_N^{\frac{1}{2}}(\Gamma)\|^2\\
    & \quad + C(N + 1)^4 \Big(\sup_{t\geq 0} \|\nabla \rho_1\|^4\Big)
    \frac{1}{\nu^{3/2}\epsilon^{3/2}} \Big[\|A^{\frac{1}{2}}
    D_N^{\frac{1}{2}} W\|^2 + \|\ADN(\Gamma)\|^2\Big].
  \end{align*}
  Hence, we get
  \begin{equation*}
    \frac{1}{2}\frac{d}{dt}\Big(\| \ADN (W)\|^2 +
    \| \ADN (\Gamma) \|^2\Big) \leq
    M\Big(\| \ADN (W)\|^2 +
    \| \ADN (\Gamma)\|^2\Big),
  \end{equation*}
  where $M=M\big(N, \nu^{-1}, \epsilon^{-1}, \|\rho_1\|_{L^\infty(0,
    T; H^1)}), \|w_1\|_{L^\infty(0, T; H^1)}\big)$. Therefore, the
  conclusion follows by using the Gronwall's inequality and recalling
  that $W(0)=0$ as well as $\Gamma(0)=0$.
\end{proof}

\section{Convergence to the solutions of \eqref{eq:Bouss-mean} when $N
  \to +\infty$} \label{limit}
As a result of the previous section, we have a disposal a sequence of
solutions $\{(w_N, \rho_N, q_N)\}_{N\in \N}$ to
\eqref{eq:Bouss-approx} that actually depends on $N$,
$\alpha$ and $\epsilon$.  Here, we prove our main result,
Theorem~\ref{main}, which states that, taking $\epsilon$
such that $\epsilon\to 0$ as $N\to +\infty$, then the sequence
$\{(w_N, \rho_N, q_N)\}_{N\in\N}$ is compact, in a suitable sense, and
converges to a solution of the mean Boussinesq
problem \eqref{eq:Bouss-mean}, as $N\to +\infty$.

\begin{proof}[Proof of Theorem~\ref{main}]
  We will show that the weak formulation
  \eqref{weak-form-1}-\eqref{weak-form-2} converges to the weak
  formulation of \eqref{eq:Bouss-mean} as $N \to +\infty$, and that
  the limiting functions satisfy the claimed properties.  We divide
  the proof into three steps.

  \indent $\mathbf{Step\, 1}$: \emph{Estimates uniform in $N$}. To get
  compactness properties about the sequences $\big\{\big(\wen,
  \ren\big) \big\}_{N\in\N}$ and $\big\{\big(D_N (\wen ), D_N
  (\ren)\big)\}_{N\in\N}$, we will provide some additional bounds.
  Using the same notation of the previous section, we quote in the
  following tables the estimates that we will use for passing to the
  limit, as $N \to +\infty$. 
% The next tables are organized as
  %\eqref{table} and \eqref{table-rho}.  
Let $I=[0, T]$, $T>0$, we have
  \begin{equation} \label{table1}
    \begin{tabular}{|@{\,}l@{\,}|@{\,}l@{\,}|@{\,}l@{\,}|@{\,}l@{\,}|}
      \hline
      Label   &   Variable & Bound   &    Order \\
      \hline \hline 
      (a)    
      & $w_N$
      &  $L^{\infty}(I; H_0)\cap L^2(I; H_1)$  
      &   $O(1)$ \\
      \hline
      (b)    
      & $w_N$
      &  $L^{\infty}(I; H_1)\cap L^2(I; H_2)$  
      &   $O(\alpha^{-1})$ \\
      \hline
      (c)    
      & $D_N(w_N)$
      &  $L^{\infty}(I; H_0)\cap L^2(I; H_1)$  
      &   $O(1)$ \\
      \hline
      (d)    
      & $\D_t w_N$
      &  $L^2(I; H_0)$  
      &   $O(\alpha^{-1})$ \\
      \hline
      (e)    
      & $\D_t D_N (w_N)$
      &  $L^{4/3}(I; H_{-1})$  
      &   $O(1) $  \\ 
      \hline
      (f) 
      & $q_N$ 
      & $L^2(I ; H^1) \cap L^{5/3}(I ; W^{2,5/3})$ 
      & $O(\alpha^{-1})$ \\
      \hline
    \end{tabular} 
  \end{equation} 
  and
  \begin{equation} \label{table1-rho}
    \begin{tabular}{|@{\,}l@{\,}|@{\,}l@{\,}|@{\,}l@{\,}|@{\,}l@{\,}|}
      \hline
      Label   &  Variable  & Bound  & Order  \\
      \hline \hline 
      (a)    
      &$\rho_N$    
      &  $L^{\infty}(I; L^2)$
      &  $O(1)$ \\
      \hline
      (b)    
      & $\rho_N$  
      &  $L^{\infty}(I; H^1)$
      &   $O(\alpha^{-1})$\\
      \hline
      (c)  
      &   $D_N(\rho_N)$  
      &  $L^{\infty}(I; L^2)$
      &   $O(1)$\\
      \hline
      (d)   
      &  $\D_t\rho_N$    
      &  $L^2(I; L^2)$ 
      &  $O(\alpha^{-1})$\\
      \hline
      (e) 
      &  $\D_tD_N (\rho_N)$   
      &  $L^{2}(I; H^{-2})$  
      &   $O(1)$\\   
      \hline
    \end{tabular} 
  \end{equation} 
  \indent $\bullet\,\, Checking$
  (\eqref{table1}-(a))-(\eqref{table1}-(d)) and
  (\eqref{table1-rho}-(a))-(\eqref{table1-rho}-(d)): These bounds are
  direct consequences of those in
  (\eqref{table}-(a))-(\eqref{table}-(d)) and
  \eqref{table-rho}-(b))-(\eqref{table-rho}-(d)), respectively.

  \indent $\bullet\,\, Checking$ (\eqref{table1}-(e)) and
  (\eqref{table1-rho}-(e)): Let be given $v\in L^4(0, T; H_1)$ and
  $h\in L^2(0, T; H^2)$.  We use $D_N(v)$ and $D_N (h)$ as test
  functions. Since $D_N$ commute with differential operators, $G$ and
  $D_N$ are self-adjoint, then classical integrations by parts give
  \begin{align}
    & \label{eq:Dt-DN-w}
    \begin{aligned}
      \big(\D_t w_N,  D_N  (v)\big) = & \big(\D_tD_N (w_N), v\big)\\
      = &\nu\big(\Delta w_N, D_N (v)\big) + \big(D_N (w_N)
      \otimes D_N (w_N), GD_N (\nabla v)\big) \\
      &+ \big(D_N (\rho_N)e_3, v\big),
    \end{aligned}\\[0.2 em]
    & \label{eq:Dt-DN-rho}
    \begin{aligned}
      \big(\D_t\rho_N, \, D_N (h)\big) & = \big(\D_tD_N (\rho_N ), h\big)\\
      & = \epsilon\big(\Delta \ren, D_N(h)\big) + \big(D_N (\rho_N )
      D_N (w_N ), GD_N (\nabla h)\big).
    \end{aligned}
  \end{align}

  Let us consider \eqref{eq:Dt-DN-w}.  Using the duality pairing
  $\langle \cdot, \cdot \rangle$ between $H^1$ and $H^{-1}$ and
  standard estimates we get
  \begin{align*}
    &|(\Delta w_N , D_N (v))| = |(\nabla D_N (w_N ),\nabla v)|
    \leq C_1(t)\|v\|_1,\\
    &| (D_N (\rho_N)e_3, v) | = |\langle D_N (\rho_N)e_3, v \rangle |
    \leq \|D_N (\rho_N)\|\|v\|_1 \leq C_2(t)\|v\|_1.
  \end{align*}
  Observe that, the $L^2(0, T; H_1)$-bound for $D_N (w_N)$ together
  with the $L^2(0, T; L^2(\T^3))$-bound for $D_N (\rho_N)$ imply that
  $C_1, C_2 \in L^2(0, T )$ with order $O(1)$.  Again, from
  (\eqref{table1}-(c)) and usual interpolation inequalities, we obtain
  that $D_N(w_N)$ belongs to $L^{8/3}(0,T; L^4(\T^3)^3)$, which yields
  \begin{align*}
    & D_N(w_N)\otimes D_N(w_N)\in L^{4/3}(0,T ; L^2(\T^3)^9).
  \end{align*}
  Hence, using the this bound together with $\|GD_N(\nabla v)\|\leq
  \|\nabla v\|$, we get
  \begin{equation*}
    |(D_N(w_N)\otimes D_N(w_N),  GD_N (\nabla v))|\leq C_3(t)\|v\|_1,
  \end{equation*}
  where $C_3\in L^{4/3}(0,T)$, and these estimates are uniform in
  $N\in \N$.  Therefore, using all the above bounds we get, uniformly
  in $N$,
  \begin{align*}
    |(\D_t D_N (w_N ), v)| \leq \big(\nu C_1(t) + C_2(t) + C_3(t)
    \big) \|v\|_1, % \textcolor{red}{ \in L^{4/3}(0, T ),}
  \end{align*}
  with $\big(\nu C_1 + C_2 + C_3 \big) \in L^{4/3}(0, T )$. This
  proves (\eqref{table1}-(e)).

  Now, take into account \eqref{eq:Dt-DN-rho}. For $h\in L^2(0, T;
  H^2)$, we have that
  \begin{align*}
    & |(\Delta \rho_N , D_N (h))| = |( D_N (\rho_N ),\Delta v)|
    \leq \|D_N (\rho_N)\|\|h\|_2 \leq C_4(t)\|h\|_2, 
    \intertext{where  $C_4\in L^2(0, T)$ with order $O(1)$.
      Since $D_N(\rho_N)D_N(w_N)\in L^2(0,T ;
      L^{3/2}(\T^3)^3)$ and $W^{2,2}\hookrightarrow W^{1, 3}$, we also
      have that} &\begin{aligned} |(D_NG\big((\rho_N) D_N(w_N)\big),
      \nabla h)| &=
      |\langle D_N G\big(D_N(\rho_N) D_N(w_N)\big),\nabla h \rangle |\\
      &\leq \|D_N G\big(D_N(\rho_N) D_N(w_N)\big)\|_{L^{3/2}(\T^3)}
      \|\nabla h\|_{L^3(\T^3)} \\
      &\leq C\|\big(D_N(\rho_N) D_N(w_N)\|_{L^{3/2}(\T^3)}
      \|\nabla h\|_{L^3(\T^3)} \\
      &\leq C\|D_N(\rho_N)\| \|D_N(w_N)\|_{H^1} \|h\|_{H^2} \\
      &\leq C_5(t)\|h\|_{H^2},
    \end{aligned}
  \end{align*}
  with $ C_5\in L^2(0, T)$ with order $O(1)$. Whence
  \begin{align*}
    |(\D_tD_N (\rho_N ), h)| &\leq (\epsilon C_4(t) + C_5(t
    ))\|h\|_{H^2}  \\
    & \leq (\sqrt{\epsilon}C_4(t) + C_5(t))
    \|h\|_{H^2}, %\textcolor{red}{ \in L^2(0, T ),}
  \end{align*}
  where in the last step we used that $0<\epsilon<1$.  Here
  $(\sqrt{\epsilon}C_4 + C_5) \in L^2(0, T)$, hence
  (\eqref{table1-rho}-(e)) follows.

  \indent $\bullet\,\, Checking$ (\eqref{table1}-(f)): To obtain
  further regularity properties of the pressure we use again the
  bounds listed in \eqref{table}. Let $A_N= −G\big( \nabla \cdot [D_N
  (w_N) \otimes D_N (w_N)]\big)$ be the bi-linear form introduced in
  the proof of Theorem~\ref{preliminary-results}.  From the estimates
  proved in the previous section we have that $\rho_N\in L^\infty(0,
  T; H^1)$ and that $A_N \in L^2([0, T]\X \T^3)^3$.  This gives the
  first bound in $L^2(0,T ; H^1)$ for $q_N$.  Moreover, classical
  interpolation inequalities used together with (\eqref{table1}-(c))
  yield $D_N (w_N ) \in L^{10/3}([0, T ] \X \T^3)^3$.  Therefore, $A_N
  \in L^{5/3}(0, T; W^{1,5/3}(\T^3)^3)$. Consequently, we have that
  $q_N$ is bounded in $L^2(0, T; H^1) \cap L^{5/3}(0, T; W^{2,5/3}(\T^3))$
  with order $O(\alpha^{-1})$.

  \indent $\mathbf{Step\, 2}$: \emph{Compact
    sub-sequence.} Thanks to the uniform
  estimates established in \eqref{table1} and \eqref{table1-rho}, one
  can use the Aubin-Lions compactness theorem (see, e.g.,
  \cite{Lions}) that provides the existence of a sub-sequence of
  $\{(\wen, \ren, q_N)\}_{N\in\N}$ (still denoted $\{(\wen, \ren,
  q_N)\}_{N\in\N}$) and functions $w$, $\rho$, $q$, $z$ and $\gamma$,
  with $\nabla\cdot w=0$ and $\nabla\cdot z=0$, such that
  \begin{align*}
    &\quad \quad\begin{aligned}
      &  w, z \in L^\infty(0, T; H_1) \cap L^2(0, T; H_2),\\
      &\rho, \gamma \in L^\infty(0, T; H^1),\\
      & q \in L^2(0, T; H^1) \cap L^{5/3}(0, T; W^{2,5/3}(\T^3)),
    \end{aligned}
    \intertext{and that} & \left.\begin{array}{ll}
        \left.\begin{array}{l} \wen \to w
          \end{array}\right.
        \left\{\begin{array}{l}
            \textrm{weakly in } L^2(0, T; H_2 ) \textrm{ and }
            \textrm{weakly}^\star \textrm{ in } L^\infty(0, T; H_1),\\
            \textrm{strongly in } L^2(0, T; H_1),
          \end{array}\right.
      \end{array} \right.\\
    &\left.\begin{array}{ll} \left.\begin{array}{l}
            \rho_N \to \rho 
          \end{array}\right.
        \left\{\begin{array}{l}
            \textrm{weakly in } L^2(0, T; H^1) \textrm{ and }
            \textrm{weakly}^\star \textrm{ in } L^\infty(0, T; H^1),\\
            \textrm{strongly in } L^2(0, T; L^2(\T^3)),
          \end{array}\right.
      \end{array} \right. \\
    &\quad\,\,  \D_t \wen \to \D_t w  \textrm{ weakly in }  
    L^2([0, T]\X \T^3), \\
    &\quad\,\,   \D_t\ren \to \D_t \rho  \textrm{ weakly in }  
    L^2([0, T]\X \T^3), \\
    \displaybreak[0]
    & \left.\begin{array}{ll}
        \left. \begin{array}{l}
            D_N (w_N ) \to z\\
          \end{array}\right. 
        \left\{ \begin{array}{l}    
            \textrm{weakly in } L^2(0, T; H_1) \textrm{ and } 
            \textrm{weakly}^\star \textrm{ in } L^\infty(0, T; H_0),\\
            \textrm{strongly in } L^p([0, T]\X \T^3), 
            \forall p \in [1, 10/3[,
          \end{array}\right.
      \end{array}\right.\\
    & \left.\begin{array}{ll} \left. \begin{array}{l} D_N (\rho_N )
            \to \gamma
          \end{array}\right. 
        \left\{ \begin{array}{l}    
            \textrm{weakly in } L^2(0, T; L^2(\T^3)) \textrm{ and } 
            \textrm{weakly}^\star \textrm{ in } L^\infty(0, T; L^2(\T^3)),\\
            \textrm{strongly in } L^2(0, T; H^{-1}), 
          \end{array}\right.
      \end{array}\right.\\
    &\quad\,\, \D_tD_N (w_N ) \to \D_t z
    \textrm{ weakly in } L^{4/3}(0,T; H_{-1}),\\
    &\quad\,\, \D_tD_N (\rho_N ) \to \D_t\gamma \textrm{ weakly in }
    L^2(0,T; H^{-2}).
  \end{align*}
  By exploiting
 the same identification of the limit used in
  \cite{Be-Lew-2012}, one can check that $z = A w$. Moreover, using
  the notation $(\cdot , \cdot )$ for the scalar product in $L^2([0,T
  ]\X \T^3)$, we have that
  \begin{equation*}
    (D_N(\rho_N), h) = (\rho_N, D_N(h)) \overset{N\to+\infty}{\longrightarrow}
    (\rho, Ah) \textrm{ for all } h\in  L^2(0,T; H^2),
  \end{equation*}
  and since $ D_N (\rho_N ) \to \gamma$ weakly in $L^2(0, T;
  L^2(\T^3))$ it follows that $(\gamma, h) = (\rho, Ah)$ and hence
  $\gamma = A\rho$, in the distributional sense.  Again, by 
the above convergences types and the interpolation theorem, we also have
  that $w\in C([0, T]; H_1)$ and $\rho\in C([0, T]; L^2(\T^3))$.

  \indent$\mathbf{Step\, 3}$: \emph{Taking the limit in the system.}  In order to
  pass to the limit in the weak formulation
  \eqref{weak-form-1}-\eqref{weak-form-2} we use the compactness
  properties satisfied by $\{D_N(w_N)\}_{N\in\N}$ and
  $\{D_N(\rho_N)\}_{N\in\N}$. Let us focus on the nonlinear terms.
  The convergence results from Step~2 imply that
  \begin{align*}
    &D_N (w_N ) \otimes D_N (w_N ) \to Aw\otimes Aw \textrm{ strongly
      in } L^p([0, T ] \X \T^3)^9, \,\, \forall p \in [1, 5/3[,\\
    \intertext{and that} &\quad \quad D_N (\rho_N ) D_N (w_N ) \to
    A\rho Aw \textrm{ in the sense of distribution}.
  \end{align*}
  We actually prove that the latter convergence is stronger:  For
  $h\in L^2(0, T; H^1)$, we have
  \begin{equation*}
    \begin{aligned}
      \Big|\int_0^T \Big( G\big( D_N(\rho_N) & D_N (w_N) - A\rho
      Aw\big),
      \nabla h \Big) \Big| \\
      \leq& \Big|\int_0^T \Big(D_N(\rho_N)\big( D_N (w_N) - Aw\big),
      \nabla G(h) \Big)\Big|\\
      &+ \Big|\int_0^T \Big(\big( D_N (\rho_N) - A\rho\big), Aw\cdot
      \nabla G(h) \Big)\Big|\\
      =& I_1^N + I_2^N.
    \end{aligned}
  \end{equation*}
  Now, $I_1^N \to 0$, as $N\to +\infty$, since $D_N(w_N) \to Aw$
  strongly in $L^2(0, T; H_0)$ and $D_N(\rho_N)$ is uniformly bounded
  in $L^\infty(0, T; L^2(\T^3))$.  Let us also prove that $I_2^N \to
  0$, as $N\to +\infty$.  
  
  First, observe that $ Aw\cdot \nabla G(h)
  \in L^2(0, T; L^2(\T^3))$. Indeed, recalling that the operator
  $(\nabla \cdot) \circ G$ makes to ``gain one derivative'', we have
  that $\nabla G(h)\in L^2(0, T; H^2)$, then using the H\"older's
  inequality and the embedding $W^{2,2}(\T^3)\hookrightarrow
  L^{\infty}(\T^3)$, we get
  \begin{align*}
    \int_0^T\int_{\T^3} |Aw|^2 |\nabla G(h)|^2&
    \leq \int_0^T \|Aw\|^2 \|\nabla G(h)\|^2_{L^\infty(\T^3)}\\
    & \leq C \int_0^T \|Aw\|^2 \|\nabla G(h)\|^2_{H^2} \\
    &\leq C \|Aw\|^2_{L^\infty(0, T; L^2(\T^3))}\|h\|^2_{L^2(0, T;
      H^1)}.
  \end{align*}
  Thus, thanks to the weak convergence of $D_N(\rho_N)\to A\rho$ in
  $L^2(0, T; L^2(\T^3))$ we have that $I_2^N\to 0$, as $N \to
  +\infty$.

  Since all the other terms in the weak equation pass easily to the
  limit in \eqref{weak-form-1} and \eqref{weak-form-2}, as $N \to
  +\infty$, we conclude the following: For all $(v, h)\in L^2(0, T;
  H_1)\X L^2(0, T; H^1)$
  \begin{align*}
    & \begin{aligned} \int_0^T \int_{\mathbb{T}_3} \D_t w & \cdot v -
      \int_0^T \int_{\mathbb{T}_3} G\big( A(w) \otimes
      A(w)\big) : \nabla v \\
      & + \nu \int_0^T \int_{\mathbb{T}_3} \nabla w : \nabla v +
      \int_0^T \int_{\mathbb{T}_3} \nabla q \cdot v =
      \int_0^T\int_{\mathbb{T}_3}\rho e_3 \cdot v,
    \end{aligned}  \\
    & \int_0^T \int_{\mathbb{T}_3} \D_t \rho \cdot h - \int_0^T
    \int_{\mathbb{T}_3} G\big( A(\rho) A(w) \big) \cdot \nabla h =0,
  \end{align*}
  which is the weak formulation for \eqref{eq:Bouss-mean}.
Hence, the conclusion follows.
\end{proof}

Finally, recasting the argument used in the proof of
\cite[Proposition~4.1]{Be-Lew-2012}, one proves the following result.

\begin{proposition} Let $u_0\in H_0$ and $\theta_0\in L^2(\T^3)$.  Let
  $(w, \rho, q)$ the weak regular solution to \eqref{eq:Bouss-mean}
  given by in Theorem~\ref{main} and let $u=Aw$ and
  $\theta=A\rho$. Then, it holds true that
  \begin{equation*}
    \frac{1}{2}\frac{d}{dt}
    \left(\|u\|^2 +
      \|\theta\|^2\right) +
    \nu\|\nabla u\|^2 \leq
    (\theta e_3, u \big).
  \end{equation*}
\end{proposition}
\begin{proof}
  Consider the energy equality \eqref{energy-equality-2} for the
  approximate model \eqref{eq:Bouss-approx}. The computations made in
  the proof of Theorem~\ref{preliminary-results}, together with the
  analysis performed in Theorem~\ref{main}, show also that
  \begin{align}
    & \left.\begin{array}{l}
        D^{\frac{1}{2}}_N(w_N ) \to A^{\frac{1}{2}} w \\
        \ADN(w_N ) \to Aw
      \end{array}\right\{ \left. \begin{array}{l}
        \textrm{weakly in } L^2(0, T; H_1)   \textrm{ and}\\
        \textrm{weakly}^\star 
        \textrm{ in }  L^\infty(0, T; H_0),
      \end{array}\right. \label{eq:utility-0}
    \\
    & \left.\begin{array}{l}
        D^{\frac{1}{2}}_N (\rho_N ) \to A^{\frac{1}{2}} \rho \\
        \ADN (\rho_N) \to A\rho
      \end{array}\right\{ \left. \begin{array}{l}
        \textrm{weakly in } L^2(0, T; L^2(\T^3)) \textrm{ and}\\
        \textrm{weakly}^\star 
        \textrm{ in }  L^\infty(0, T; L^2(\T^3)),
      \end{array}\right.  \label{eq:utility-00}
  \end{align}
  Further, we also have that $\D_t w_N\in L^{4/3}(0, T; H_{-1})$. In fact,
  taking $v\in L^4(0, T; H_1)$ and using $\ADN(v)$ as test function,
  we reach
  \begin{equation*} %\label{eq:utility-1}
    \begin{aligned}
      (\D_t w_N,  \ADN (v)) =& (\D_t\ADN (w_N), v)\\
      =& \nu\big(\Delta w_N, \ADN (v)) - (D_N (w_N) \otimes D_N
      (w_N), G\ADN (\nabla v)\big)\\
      &+\big (D_N^{\frac{1}{2}} (\rho_N)e_3, A^{\frac{1}{2}}v\big).
    \end{aligned}
  \end{equation*}
  Thus, using % the duality pairing $\langle \cdot, \cdot \rangle$
  %between $H^1$ and $H^{-1}$ and 
standard estimates we get
  \begin{equation} \label{eq:DtADN-estimate}
    \begin{aligned}
      &|(\Delta w_N , \ADN (v))| = |(\nabla \ADN (w_N ),\nabla v)|
      \leq C_1(t)\|v\|_1,\\
      &| (D_N^{\frac{1}{2}} (\rho_N)e_3, A^{\frac{1}{2}}v) |
      \leq \|D_N^{\frac{1}{2}} (\rho_N)\|\|v\|_1 \leq C_2(t)\|v\|_1,\\
      &|(D_N(w_N)\otimes D_N(w_N), G\ADN (\nabla v))|\leq
      C_3(t)\|v\|_1,
    \end{aligned}
  \end{equation}
  where $C_1, C_2\in L^2(0,T)$ and $C_3\in L^{4/3}(0,T)$, and these
  estimates are uniform in $N\in \N$.  Therefore,
  using %\eqref{eq:utility-1} and
  \eqref{eq:DtADN-estimate} we obtain %get, uniformly in $N$,
  \begin{equation} \label{eq:Dt-ADN-nuova} |(\D_t \ADN (w_N ), v)|
    \leq \big(\nu C_1(t) + C_2(t) + C_3(t) \big) \|v\|_1,
  \end{equation}
  with $\big(\nu C_1 + C_2 + C_3 \big) \in L^{4/3}(0, T )$.
  By appealing to the Aubin-Lions theorem, thanks to \eqref{eq:utility-0}
  and \eqref{eq:Dt-ADN-nuova}, we also have
  \begin{equation} \label{eq:utility-2} \ADN (w_N) \to Aw \textrm{
      strongly in } L^2(0, T; H_0).
  \end{equation}
  
  Now, as a consequence of \eqref{eq:utility-00} and
  \eqref{eq:utility-2} we get, for $t\in [0, T]$
  \begin{equation} \label{eq:bi-lin-conv} \int_0^t \Big(\ADN
    (\rho_N)e_3, \ADN (w_N)\Big)ds\, \overset{N\to
      +\infty}{\longrightarrow} \, \int_0^t (A\rho e_3, A w) ds.
  \end{equation} 
  % \textcolor{red}{ \marginpar{\textcolor{red}{forse questa parte
  % posso tagliarla.}}
  In fact, we have that
  \begin{equation*}
    \begin{aligned}
      \bigg\vert \int_0^t & \Big(\ADN (\rho_N)e_3, \ADN (w_N)\Big)ds -
      \int_0^t (A\rho
      e_3, A w) ds \bigg\vert\\
      &\quad \quad\leq \left\vert \int_0^t \Big((\ADN (\rho_N)- A\rho
        )e_3, Aw
        \Big)ds\right\vert \\
      &\hspace{3.5 cm} + \left\vert \int_0^t \Big((\ADN
        (\rho_N)e_3, \ADN (w_N) - Aw)\Big)ds \right\vert\\
      &\quad \quad \leq \left\vert \int_0^t \Big((\ADN (\rho_N)- A\rho
        )e_3, Aw
        \Big)ds\right\vert \\
      &\hspace{3.2 cm} + \|\ADN (\rho_N)\|_{L^2(0, T; L^2)}^2 \|\ADN
      (w_N) - Aw\|_{L^2(0, T; H_0)}^2.
    \end{aligned}
  \end{equation*}
  Since the sequence $\{\ADN (\rho_N)\}_{N\in\N}$ is uniformly bounded
  in $L^2(0, T, L^2(\T^3))$ and $Aw\in L^2(0, T; H_0)$, then
  \eqref{eq:utility-00} and \eqref{eq:utility-2} allow us to pass to
  the limit in the above estimate as $N\to+\infty$, and this proves
  \eqref{eq:bi-lin-conv}.
  % }

  Recalling that for every $N \in \N$, it holds that $w_N (0) =
  G(u_0)=\overline{u}_0 \in H_2$ and that $\rho_N (0) =
  G(\theta_0)=\overline{\theta}_0 \in H^2$, then taking the limit as
  $N\to+\infty$ in the right-hand side of \eqref{energy-equality-2} we
  get
  \begin{align*}
    \frac{1}{2} \Big(\|\ADN (w_N )(0)\|^2 + &\|\ADN (\rho_N
    )(0)\|^2\Big)
    + \int_0^t \Big(\ADN (\rho_N)e_3,\ADN (w_N)\big)ds \Big) \\
    & \overset{N\to+\infty}{\longrightarrow} \frac{1}{2}
    \big(\|Aw(0)\|^2 + \|A\rho(0)\|^2\big) + \int_0^t (A\rho e_3, A w)
    ds,
  \end{align*}
  whence
  \begin{align*}
    \underset{N\to+\infty}{\lim \sup} & \frac{1}{2} \left(\|\ADN (w_N
      )(t)\|^2 + \|A^{\frac{1}{2}} D^{\frac{1}{2}}_N (\rho_N )(t)\|^2\right) \\
    &+\underset{N\to+\infty}{\lim \inf} \left( \nu\int_0^T\|\nabla
      \ADN (w_N)(s)\|^2ds + \epsilon\int_0^T\|\nabla A^{\frac{1}{2}}
      D_N^{\frac{1}{2}}\rho_N(s)\|^2ds \right)\\
    & \leq \frac{1}{2} (\|Aw(0)\|^2 + \|A\rho(0)\|^2) + \int_0^t
    (A\rho e_3, A w) ds.
  \end{align*}
  By lower semicontinuity of the norm and identification of the weak
  limit, we get the thesis.
\end{proof}


\begin{thebibliography}{1}
\bibitem{Adam-Stolz} N.\, A.\ Adams, S.\ Stolz, \emph{Deconvolution
    methods for subgrid-scale approximation in large eddy simulation,
    in: Modern Simulation Strategies for Turbulent Flow},
  R.T. Edwards, 2001.

\bibitem{Be-Ili-Lay-2006} L.\ C.\ Berselli,\ T.\ Iliescu,\ W.\ J.\
  Layton, \emph{Mathematics of Large Eddy Simulation of Turbulent
    Flows}, Scientific Computation, Springer-Verlag, Berlin, 2006.

\bibitem{Be-Lew-2012} L.\ C.\ Berselli, R.\ Lewandowski,
  \emph{Convergence of approximate deconvolution models to the mean
    Navier-Stokes equations}, Ann. Inst. H. Poincar\'e Anal. Non
  Lin\'eaire 29 (2012) 171-198.

\bibitem{Be-Ca-Lew-2013} L.\ C.\ Berselli, D.\ Catania\ R.\
  Lewandowski, \emph{Convergence of approximate deconvolution models
    to the mean magnetohydrodynamics equations: Analysis of two
    models} J. Math. Anal. Appl. 401 (2013) 864-880.
  
\bibitem{Be-Spi-2011} L.\, C.\, Berselli, S.\, Spirito, \emph{On the
    Boussinesq system: regularity criteria and singular limits},
  Methods Appl. Anal. 18 (2011) 391-416

\bibitem{Cha-Re-2013} T.\ Chac\'on-Rebollo, R.\ Lewandowski,
  \emph{Mathematical and numerical foundations of turbulence models},
  Birkh\"auser, New-York, 2013 (in press).

\bibitem{Do-Gib-1995} C.\ R.\ Doering, J.\ D.\ Gibbon, \emph{Applied
    analysis of the Navier–Stokes equations}, Cambridge Texts in
  Applied Mathematics, Cambridge University Press, Cambridge, 1995.

\bibitem{Fan-Zhou-2009} J.\, Fan, Y.\, Zhou. \emph{A note on
    regularity criterion for the 3D Boussinesq system with partial
    viscosity.}  Appl. Math. Lett. 22 (2009), 802-805.

\bibitem{Fan-Zhou} J.\ Fan, Y.\ Zhou, \emph{On the Cauchy problems for
    certain Boussinesq-$\alpha$ equations} Proceedings of the Royal
  Society of Edinburgh: Section A Mathematics, Vol. 140, Issue 02
  April 2010, pp 319-327

\bibitem{Guo-1995} B.\ Guo, \emph{Nonlinear Galerkin methods for
    solving two dimensional Newton-Boussinesq equations},
  Chin. Ann. Math., Ser. B 16 (1995), no. 3, 379-390.
  
\bibitem{Lew-2009} R.\ Lewandowski, \emph{On a continuous
    deconvolution equation for turbulence models}, Lecture Notes of
  Ne\c{c}as Center for Mathematical Modeling 5 (2009) 62-102

\bibitem{Lions} J.-L.\, Lions, \emph{Quelques m\'ethodes de
    r\'esolution des probl\`emes aux limites non lin\'eaires}, Dunod,
  Gauthier-Villars, Paris, 1969.

\bibitem{Majda-2003} A.\, Majda, \emph{Introduction to PDEs and Waves
    for the Atmosphere and Ocean}, Courant Lecture Notes in
  Mathematics, vol. 9, AMS/CIMS, 2003.

\bibitem{McWill-2006} J.\, C.\, McWilliams \emph{Fundamentals of
    Geophysical Fluid Dynamics}, Department of Atmospheric and Oceanic
  Sciences University of California, Los Angeles, 2006

\bibitem{Sal-1998} R.\, Salmon,\, \emph{Lectures on geophysical fluid
    dynamics.}  Oxford University Press, New York, 1998.

\bibitem{Sag-2001} P.\ Sagaut, \emph{Large Eddy Simulation for
    Incompressible Flows}, Springer-Verlag, Berlin, 2001.

\bibitem{Selmi-2012} R.\, Selmi \emph{Global Well-Posedness and
    Convergence Results for the 3D-Regularized Boussinesq System}
  Canad. J. Math. Vol. 64 (6), 2012 pp. 1415-1435

\bibitem{St-Ad-1999} S.\, Stolz, N.\, A.\, Adams, \emph{An approximate
    deconvolution procedure for large-eddy simulation}, Phys. Fluids
  11 (7) (1999) 1699-1701.

\bibitem{St-Ad-2001} S.\, Stolz, N.\, A.\, Adams, L.\, Kleiser,
  \emph{An approximate deconvolution model for large-eddy simulation
    with application to incompressible wall-bounded flows},
  Phys. Fluids 13 (4) (2001) 997-1015.
\end{thebibliography}
\end{document}